\documentclass[a4paper,11pt]{article}
\usepackage{amsmath, amsfonts, amscd, amssymb, amsthm, enumerate, xypic}

\def\defterm{\textbf}

\newcommand{\Mat}{\operatorname{M}}
\newcommand{\charac}{\operatorname{char}}

\newcommand{\id}{\operatorname{id}}
\newcommand{\GL}{\operatorname{GL}}
\newcommand{\Ker}{\operatorname{Ker}}

\newcommand{\im}{\operatorname{Im}}

\newcommand{\tr}{\operatorname{tr}}

\newcommand{\rk}{\operatorname{rk}}

\renewcommand{\setminus}{\smallsetminus}

\def\F{\mathbb{F}}
\def\K{\mathbb{K}}

\def\N{\mathbb{N}}

\def\lcro{\mathopen{[\![}}
\def\rcro{\mathclose{]\!]}}

\theoremstyle{definition}
\newtheorem{Def}{Definition}[section]

\theoremstyle{plain}
\newtheorem{theo}{Theorem}[section]
\newtheorem{prop}[theo]{Proposition}
\newtheorem{cor}[theo]{Corollary}
\newtheorem{lemma}[theo]{Lemma}

\theoremstyle{plain}

\theoremstyle{remark}
\newtheorem{Rems}{Remarks}
\newtheorem{Rem}[Rems]{Remark}

\title{A note on sums of three square-zero matrices}
\author{Cl\'ement de Seguins Pazzis\footnote{Universit\'e de Versailles Saint-Quentin-en-Yvelines, Laboratoire de Math\'ematiques
de Versailles, 45 avenue des Etats-Unis, 78035 Versailles cedex, France}
\footnote{e-mail address: dsp.prof@gmail.com}}

\begin{document}

\thispagestyle{plain}

\maketitle

\begin{abstract}
It is known that every complex trace-zero matrix is the sum of four square-zero matrices,
but not necessarily of three such matrices. In this note, we prove that for every trace-zero matrix
$A$ over an arbitrary field, there is a non-negative integer $p$ such that the extended matrix $A \oplus 0_p$ is the sum of three square-zero matrices
(more precisely, one can simply take $p$ as the number of rows of $A$).
Moreover, we demonstrate that if the underlying field has characteristic $2$ then
every trace-zero matrix is the sum of three square-zero matrices.
We also discuss a counterpart of the latter result for sums of idempotents.
\end{abstract}



% Relecture 3/4 achevée.

\section{Introduction}

\subsection{The problem}

Let $\F$ be a field, and denote by $\Mat_n(\F)$ the algebra of all $n$ by $n$ square matrices with entries in $\F$.
A matrix $A \in \Mat_n(\F)$ is called \textbf{square-zero} if $A^2=0$.
Obviously, every square-zero matrix has trace zero, and hence the sum of finitely many square-zero matrices
has trace zero. Conversely, by using the Jordan canonical form it is easy to split
every nilpotent matrix into the sum of two square-zero matrices.
On the other hand, it is well known that every trace-zero matrix is a sum of nilpotent matrices:
More precisely, it is the sum of two such matrices if non-scalar, otherwise
it is the sum of three nilpotent matrices (see e.g.\ \cite{Breaz}). Hence, every trace-zero matrix is the sum of at most six square-zero matrices,
and of at most four such matrices if non-scalar. In particular, over a field
of characteristic zero, every trace-zero matrix is the sum of four square-zero matrices.
The result still holds over general fields, see Section \ref{sumoftwosection} of the present article for a proof.

On the other hand, the classification of matrices that are the sum of two square-zero matrices is known
\cite{Bothasquarezero}. Thus, it only remains to understand which matrices split into the sum
of three square-zero ones. Since the set of all square-zero $n$ by $n$ matrices is stable under similarity,
in theory one should be able to detect such matrices from their rational canonical form.
Unfortunately, in practice such a classification seems out of reach. For example, it has been shown
by Wang and Wu \cite{WangWu} that if $A \in \Mat_n(\F)$ is the sum of three square-zero matrices and $\F$ is the field of complex numbers, then
$\rk (A-\lambda I_n) \geq \frac{n}{4}$ for every non-zero scalar $\lambda$ (the result can be generalized to an arbitrary field with characteristic not $2$); Yet, it has been demonstrated in \cite{Takahashi} that there are trace-zero matrices that satisfy this condition without being
the sum of three square-zero matrices.

In this article, we shall frame the problem in a slightly different manner.
If we have a trace-zero matrix $A$, maybe $A$ is not the sum of three square-zero matrices, but can
we obtain such a decomposition by enlarging $A$? More precisely, can we find a positive integer $p$ such that
the block-diagonal matrix $A \oplus 0_p$, that is $\begin{bmatrix}
A & 0_{n \times p} \\
0_{p \times n} & 0_p
\end{bmatrix}$,
is the sum of three square-zero matrices? The motivation for studying that problem
stems from the infinite-dimensional setting. In a future article, its solution will help us characterize the endomorphisms
of an infinite-dimensional vector space that can be decomposed as the sum of three square-zero endomorphisms.

The main strategy consists in an adaptation of the methods that were used in \cite{dSPidempotentLC}
to prove that every matrix is a linear combination of three idempotent ones. In the last section of the article,
we shall prove a variation of this very result, motivated again by the case of infinite-dimensional spaces.

\subsection{Main results}

Here are our three main results.

\begin{theo}\label{sumoffour}
Let $\F$ be an arbitrary field, and $A \in \Mat_n(\F)$ be a square matrix with trace zero.
Then, $A$ is the sum of four square-zero matrices.
\end{theo}

\begin{theo}\label{allfieldssquarezero}
Let $\F$ be an arbitrary field, and $A \in \Mat_n(\F)$ be a square matrix with trace zero.
Then, $A\oplus 0_n$ is the sum of three square-zero matrices of $\Mat_{2n}(\F)$.
\end{theo}

\begin{theo}\label{carac2squarezero}
Let $\F$ be a field with characteristic $2$. Every trace-zero matrix of $\Mat_n(\F)$
is the sum of three square-zero matrices.
\end{theo}

With a similar method, we shall obtain the following theorem:

\begin{theo}\label{carac2idempotents}
Let $\F$ be an arbitrary field with characteristic $2$, and $A \in \Mat_n(\F)$ be a square matrix whose trace belongs to $\{0,1\}$.
Then, $A\oplus 0_n$ is the sum of three idempotent matrices of $\Mat_{2n}(\F)$.
\end{theo}

It is known that over $\F_2$ every matrix is the sum of three idempotent ones (see Theorem 1 of \cite{dSPidempotentLC}).
This fails over larger fields with characteristic $2$ even if we restrict to trace-zero matrices:
Take indeed such a field $\F$ together with a scalar $\lambda \in \F \setminus \{0,1\}$.
Then, $\lambda\,I_{2n}$ has trace zero yet it is not the sum of three idempotent matrices (see the remark in the middle of page 861 of \cite{dSPidempotentsums}).

The motivation for the above results stems from the following corollaries, which will be proved in Section \ref{infinitedim}.

\begin{cor}\label{3squarezeroCor}
Let $V$ be an infinite-dimensional vector space over $\F$, and $u$ be a finite-rank endomorphism with trace zero.
Then, $u$ is the sum of three square-zero endomorphisms of $V$.
\end{cor}

\begin{cor}\label{3squarezeroCorcar2}
Assume that $\F$ has characteristic $2$.
Let $V$ be an infinite-dimensional vector space over $\F$, $\alpha \in \F$ and $u$ be a finite-rank endomorphism of
$V$ whose trace belongs to $\{0,\alpha\}$. Then, $\alpha\,\id_V+u$ is the sum of three square-zero endomorphisms of $V$.
\end{cor}

\begin{cor}\label{3idemCorcar2}
Assume that $\F$ has characteristic $2$.
Let $V$ be an infinite-dimensional vector space over $\F$, $\alpha \in \{0_\F,1_\F\}$, and
$u$ be a finite-rank endomorphism of $V$ whose trace belongs to $\{0_\F,1_\F\}$.
Then, $\alpha\,\id_V+u$ is the sum of three idempotent endomorphisms of $V$.
\end{cor}

\subsection{Two basic remarks}

Throughout the article, we shall systematically use the following remarks in which,
given monic polynomials $p_1,\dots,p_d$ over $\F$, a \textbf{$(p_1,\dots,p_d)$-sum} is a square matrix $M$ that splits into
$M=A_1+\cdots+A_d$ where $p_i(A_i)=0$ for all $i \in \lcro 1,d\rcro$. In particular,
the $(t^2,t^2,t^2)$-sums are the sums of three square-zero matrices.

\begin{Rem}\label{remarksim}
If $M$ is a $(p_1,\dots,p_d)$-sum then so is any matrix that is similar to $M$.
\end{Rem}

\begin{Rem}\label{remarkoplus}
Assume that $M$ splits into a block-diagonal matrix $M=M_1 \oplus \cdots \oplus M_r$.
If $M_1,\dots,M_r$ are all $(p_1,\dots,p_d)$-sums then so is $M$.
\end{Rem}

\subsection{Structure of the article}

In short, the method here is largely similar to the one that was used in \cite{dSPidempotentLC}
to prove that every matrix is a linear combination of three idempotents.
In order to prove Theorem \ref{allfieldssquarezero}, the idea consists in showing that
$A \oplus 0_n$ is similar to a block-diagonal matrix of type $N \oplus \alpha I_{2r} \oplus 0_q$ in which $\alpha \in \F \setminus \{0\}$, $2r \leq q$ and
$N$ is \emph{well-partitioned} (see Section \ref{wellpartsection}).
By subtracting a well-chosen square-zero matrix to $N$, one is able to obtain a cyclic matrix whose minimal polynomial
can be chosen among the monic polynomials with the same degree and trace as the characteristic polynomial of $N$.
Then, thanks to the classification of sums of two square-zero matrices (see \cite{Bothasquarezero} or the appendix to this article),
it will follow without much effort that $N \oplus \alpha I_{2r} \oplus 0_q$ is the sum of three square-zero matrices.

Similar strategies will be used to prove Theorems \ref{carac2squarezero} and \ref{carac2idempotents}.

The main tools for the proofs of the above theorems are laid out in Section \ref{lemmasection}:
in there, we recall some useful notation and facts on cyclic matrices, we define and quickly
study the notion of a well-partitioned matrix, and we give a review of the characterization of matrices
that split into the sum of two square-zero matrices (in that paragraph, we include a proof of Theorem \ref{sumoffour}).

In Section \ref{SectionSquareZeroCarnot2}, we shall prove Theorem \ref{allfieldssquarezero}
over fields with characteristic not $2$. In Section \ref{SectionSquareZeroCar2}, we shall
prove Theorem \ref{carac2squarezero}, thereby completing the proof of Theorem \ref{allfieldssquarezero} over fields with characteristic $2$.
Theorem \ref{carac2idempotents} is proved in Section \ref{SectionIdempotentCar2}.
Section \ref{SectionIdempotentLC} is devoted to a variation of Theorem 1 of \cite{dSPidempotentLC}
on the linear combinations of three idempotent matrices.
In Section \ref{infinitedim}, we derive Corollaries \ref{3squarezeroCor}, \ref{3squarezeroCorcar2}
and \ref{3idemCorcar2} from Theorems \ref{allfieldssquarezero}, \ref{carac2squarezero} and \ref{carac2idempotents}
respectively, and we prove a similar result for linear combinations of idempotents.

The appendix consists of a simplified proof of Botha's characterization of sums of two square-zero matrices (Theorems 1 and 2 from \cite{Bothasquarezero}).

\section{Some useful lemmas, and additional notation}\label{lemmasection}

In this section, we recall some basic results from \cite{dSPidempotentLC}.

\subsection{Additional notation}

Throughout the article, we choose an algebraic closure of $\F$ and denote it by $\overline{\F}$.

Similarity of two matrices $A$ and $B$ of $\Mat_n(\F)$ will be written $A\simeq B$.

\vskip 2mm
The characteristic polynomial of a square matrix $M$ will be denoted by $\chi_M$, its trace by $\tr (M)$.

\vskip 2mm
Let $p(t)=t^n-\underset{k=0}{\overset{n-1}{\sum}}a_k\,t^k \in \F[t]$ be a monic polynomial with degree $n$.
Its \defterm{companion matrix} is defined as
$$C\bigl(p(t)\bigr):=\begin{bmatrix}
0 &   & & (0) & a_0 \\
1 & 0 & &   & a_1 \\
0 & \ddots & \ddots & & \vdots \\
\vdots & \ddots & & 0 & a_{n-2} \\
(0) & \cdots & 0 &  1 & a_{n-1}
\end{bmatrix}.$$
The characteristic polynomial of $C(p(t))$ is precisely $p(t)$, and so is its minimal polynomial.
We define the trace of $p(t)$ as $\tr p(t):=\tr C\bigl(p(t)\bigr)=a_{n-1.}$

\vskip 2mm
A matrix $A \in \Mat_n(\F)$ is called \textbf{cyclic} when
$A \simeq C\bigl(p(t)\bigr)$ for some monic polynomial $p(t)$ (and then $p(t)=\chi_A(t)$).
A \textbf{good cyclic} matrix is a matrix of the form
$$A=\begin{bmatrix}
a_{1,1} & a_{1,2}  & \cdots &  & a_{1,n} \\
1 & a_{2,2} & &   &  \\
0 & \ddots & \ddots & & \vdots \\
\vdots & \ddots & \ddots & a_{n-1,n-1} & a_{n-1,n} \\
(0) & \cdots & 0 &  1 & a_{n,n}
\end{bmatrix}$$
with no specific requirement on the $a_{i,j}$'s for $j \geq i$.
We recall that such a matrix is always cyclic.

\vskip 2mm
Finally, we denote by $H_{n,p}$ the matrix unit
$\begin{bmatrix}
0 & \cdots & 0 &  1 \\
\vdots & & \vdots & (0) \\
0 & \cdots & 0 & 0
\end{bmatrix}$ of $\Mat_{n,p}(\F)$, the set of all $n$ by $p$ matrices with entries in $\F$. 

\subsection{Two basic lemmas}

The following lemma was proved in \cite{dSPidempotentLC} (see Lemma 11):

\begin{lemma}[Choice of polynomial lemma]\label{cyclicfit}
Let $A \in \Mat_n(\F)$ and $B \in \Mat_m(\F)$ be good cyclic matrices, and
$p(t)$ be a monic polynomial of degree $n+m$ such that $\tr p(t)=\tr(A)+\tr(B)$. \\
Then, there exists a matrix $D \in \Mat_{n,m}(\F)$ such that
$$\begin{bmatrix}
A & D \\
H_{m,n} & B
\end{bmatrix} \simeq C\bigl(p(t)\bigr).$$
\end{lemma}

\noindent
The second lemma we will need is folklore (it can be seen as an easy
corollary to Roth's theorem, see \cite{Roth}):

\begin{lemma}
Let $A \in \Mat_n(\F)$, $B \in \Mat_p(\F)$, and $C \in \Mat_{n,p}(\F)$.
Assume that $\chi_A$ and $\chi_B$ are coprime. Then,
$$\begin{bmatrix}
A & C \\
0 & B
\end{bmatrix} \simeq \begin{bmatrix}
A & 0 \\
0 & B
\end{bmatrix}.$$
\end{lemma}

\subsection{Well-partitioned matrices}\label{wellpartsection}

\begin{Def}
A square matrix $M$ is called \textbf{well-partitioned} if there are positive integers $r$ and $s$
and monic polynomials $p_1,\dots,p_r,q_1,\dots,q_s$ such that:
\begin{enumerate}[(i)]
\item $M=C(p_1) \oplus \cdots \oplus C(p_r) \oplus C(q_1) \oplus \cdots \oplus C(q_s)$;
\item $\deg p_i \geq 2$ for all $i \in \lcro 2,r\rcro$;
\item $\deg q_j \geq 2$ for all $j \in \lcro 1,s-1\rcro$;
\item Each polynomial $p_i$ is coprime to each polynomial $q_j$.
\end{enumerate}
Note that the polynomials $p_1,\dots,p_r,q_1,\dots,q_s$ are then uniquely determined by $M$
(beware that in (i) we really require an equality and not just a similarity).

If in addition at most one of $p_1$ and $q_s$ has degree $1$, we say that $M$ is \textbf{very well-partitioned}.
\end{Def}

We convene that the void matrix (i.e.\ the $0$ by $0$ matrix) is well-partitioned.

We will need three lemmas on well-partitioned matrices. The first one was already proved in
the course of the proof of Lemma 14 of \cite{dSPidempotentLC}:

\begin{lemma}\label{verywellpartitioned-idempotent}
Let $M \in \Mat_n(\F)$ be a very well-partitioned matrix.
For every monic polynomial $R$ with degree $n$ such that $\tr(R) \neq \tr(M)$, there
exists an idempotent matrix $P \in \Mat_n(\F)$ and a scalar $\lambda$ such that $M-\lambda P \simeq C(R)$.
\end{lemma}

By checking the details of the proof of Lemma 14 of \cite{dSPidempotentLC}, one sees that
the following result was also obtained:

\begin{lemma}\label{wellpartitioned-idempotent}
Let $M \in \Mat_n(\F)$ be a well-partitioned matrix with $m$ associated polynomials.
Let $\lambda \in \F \setminus \{0\}$.
For every monic polynomial $R$ with degree $n$ such that $\tr(R)=\tr(M)-(m-1)\lambda$, there
exists an idempotent matrix $P \in \Mat_n(\F)$ such that $M-\lambda P \simeq C(R)$.
\end{lemma}

Here, we shall require the counterpart of the preceding result for square-zero matrices.

\begin{lemma}\label{wellpartitioned-squarezero}
Let $A \in \Mat_n(\F)$ be a well-partitioned matrix, and
$R$ be a monic polynomial with degree $n$ such that $\tr(R)=\tr(A)$.
Then, there exists a square-zero matrix $N \in \Mat_n(\F)$ such that $A-N \simeq C(R)$.
\end{lemma}

The proof is an easy adaptation of the one of Lemma 14 of \cite{dSPidempotentLC}:
we give it for the sake of completeness.

\begin{proof}[Proof of Lemma \ref{wellpartitioned-squarezero}]
Denote by $p_1,\dots,p_r,q_1,\dots,q_s$ the polynomials associated with the well-partitioned matrix $A$,
and by $n_1,\dots,n_r,m_1,\dots,m_s$ their respective degrees.
Set
$$S:=\begin{bmatrix}
0_{n_1 \times n_1} & 0_{n_1 \times n_2} &  & & & & & (0) \\
-H_{n_2,n_1} & 0_{n_2 \times n_2} & \ddots & \\
(0) & \ddots & \ddots &   \\
 &  & -H_{n_r,n_{r-1}} & 0_{n_r \times n_r} & 0_{n_r \times m_1} &  \\
 & & & -H_{m_1,n_r} & 0_{m_1 \times m_1} & 0_{m_1 \times m_2} & &  \\
\vdots & & & & -H_{m_2,m_1} & 0_{m_2 \times m_2} & &  \\
 & & & &   & \ddots & \ddots &  & \\
(0) & & & \cdots &   &   & -H_{m_{s},m_{s-1}} & 0_{m_s \times m_s}
\end{bmatrix}.$$
One checks that $S$ is square-zero (this uses the fact that $n_2 \geq 2,\dots,n_r \geq 2,m_1 \geq 2,\dots,m_{s-1} \geq 2$).
Moreover, by setting $a=\underset{k=1}{\overset{s}{\sum}}m_k$ and $b=\underset{k=1}{\overset{r}{\sum}}n_k$, we find that
there are good cyclic matrices $M'_1 \in \Mat_b(\F)$ and $M'_2 \in \Mat_a(\F)$ such that
$$A-S=\begin{bmatrix}
M'_1 & 0_{b \times a} \\
H_{a,b} & M'_2
\end{bmatrix}.$$
On the other hand, $\tr(A-S)=\tr(A)=\tr(M'_1)+\tr(M'_2)$.
Lemma \ref{cyclicfit} yields a matrix $D \in \Mat_{b,a}(\F)$ such that
$$\begin{bmatrix}
M'_1 & D \\
H_{a,b} & M'_2
\end{bmatrix} \,\simeq\, C(R).$$
However, $A=A_1 \oplus A_2$ with $A_1 \in \Mat_b(\K)$ and $A_2 \in \Mat_a(\K)$ that have coprime characteristic polynomials.
It follows that
$$A \simeq A':=\begin{bmatrix}
A_1 & D \\
0 & A_2
\end{bmatrix}.$$
However,
$$A'-S=\begin{bmatrix}
M'_1 & D \\
H_{a,b} & M'_2
\end{bmatrix} \,\simeq\, C(R).$$
We conclude that there exists a square-zero matrix $S'$ that is similar to $S$ and such that
$A-S' \simeq C(R)$.
\end{proof}

\subsection{A review of sums of two square-zero matrices}\label{sumoftwosection}

The classification of sums of two square-zero matrices was completed in \cite{WangWu} for
the field of complex numbers  and in \cite{Bothasquarezero} for general fields.
We state the result:

\begin{theo}\label{sumoftwotheo}
A matrix of $\Mat_n(\F)$ can be decomposed as the sum of two square-zero matrices if and only if
all its invariant factors are odd or even polynomials.
\end{theo}

In particular, every nilpotent matrix is the sum of two square-zero matrices.

\begin{cor}\label{similaropposite}
Assume that $\charac(\F) \neq 2$.
Let $M \in \Mat_n(\F)$. Then, $M$ is the sum of two square-zero matrices
if and only if $M$ is similar to $-M$.
\end{cor}

\begin{cor}
Assume that $\charac(\F) = 2$.
Let $M \in \Mat_n(\F)$. Then, $M$ is the sum of two square-zero matrices if and only if
all the Jordan cells of $M$ corresponding to the non-zero eigenvalues (in some algebraic closure of $\F$)
are even-sized.
\end{cor}

As a corollary, we are now able to prove Theorem \ref{sumoffour}:

\begin{proof}[Proof of Theorem \ref{sumoffour}]
Let $A$ be a trace-zero matrix of $\Mat_n(\F)$.
The case when $A=0$ is obvious and we discard it from now on.

Assume first that $A$ is not a scalar matrix. Then, it is similar to a matrix $A'$ with diagonal zero (see \cite{Fillmore}),
and hence it splits into the sum of two nilpotent matrices  each of which is the sum of two square-zero matrices (one splits $A'$ into the sum of a strictly upper-triangular matrix and a strictly lower-triangular matrix).

Now, assume that $A$ is a non-zero scalar matrix. Without loss of generality, we can assume that $A=I_n$.
Then, as $\tr A=0$ we get that $\F$ has positive characteristic $p$ and that $p$ divides $n$.
Set $r(t):=t^p-t$ if $p$ is odd, otherwise set $r(t):=t^2$.
In any case, both $r(t)$ and $r(t-1)$ are even or odd, and hence
both matrices $C(r(t))$ and $C(r(t-1))$ are sums of two square-zero matrices.
Yet, $I_p+C(r(t)) \simeq C(r(t-1))$, and hence $I_p+C(r(t))$ is the sum of two square-zero matrices.
We conclude that $I_p$ is the sum of four square-zero matrices.
Since $A=I_p \oplus \cdots \oplus I_p$ (with $\frac{n}{p}$ copies of $I_p$), we conclude that $A$ is the sum of four square-zero matrices.
\end{proof}

The last basic result that we will need deals with the sum of an idempotent matrix and a square-zero one
(see \cite{dSPsumoftwotriang} for general results on the matter).

\begin{lemma}\label{proj+squarezero}
Let $n \geq 1$ and $p(t)\in \F[t]$ be a monic polynomial with degree $n$.
Let $A$ be a rank $n$ idempotent matrix of $\Mat_{2n}(\F)$.
Set $q:=p\bigl(t(t-1)\bigr)$.
Then, there exists a square-zero matrix $S$ such that
$$A-S \simeq C(q(t)).$$
\end{lemma}

\begin{proof}
Without loss of generality, we can assume that
$$A=\begin{bmatrix}
I_n & 0_n \\
I_n & 0_n
\end{bmatrix}.$$
Set then
$$S:=\begin{bmatrix}
0_n & -C(p) \\
0_n & 0_n
\end{bmatrix}.$$
Obviously, $S^2=0$. By Lemma 7 of \cite{dSPsumoftwotriang}, we get $A-S \simeq C(q(t))$.
\end{proof}

\section{Sums of three square-zero matrices over a field with characteristic not $2$}\label{SectionSquareZeroCarnot2}

Our aim is to prove Theorem \ref{allfieldssquarezero} over fields with characteristic not $2$.
Throughout the section, $\F$ denotes such a field.

Let $A \in \Mat_n(\F)$ be with trace zero. Let us consider the matrix $M:=A \oplus 0_n$
and prove that $M$ is the sum of three square-zero matrices. The basic idea is to find a well-chosen
square-zero matrix $S$ such that $M-S$ is the sum of two square-zero matrices.
Note first that replacing $M$ with a similar matrix leaves our problem invariant.
Note also that $\tr(M)=\tr(A)=0$.

First of all, we shall put $M$ into a simpler form that involves a well-partitioned matrix:

\begin{lemma}\label{WPlemma}
Let $M \in \Mat_{2n}(\F)$. Assume that $M$ has at least $n$ Jordan cells of size $1$ for the eigenvalue $0$.
Then, there exist non-negative integers $p,q,r$, a matrix $N \in \Mat_p(\F)$ and a non-zero scalar $\alpha \in \F \setminus \{0\}$
such that
$$M\simeq N \oplus \alpha\,I_{2q}\oplus 0_{r}, \quad r \geq 2q,$$
and $N$ is either nilpotent or well-partitioned.
\end{lemma}

\begin{proof}
Note that the result is obvious if $M$ is nilpotent (it suffices to take $N=M$, $\alpha=1$ and $q=r=0$ in that case).
Hence, in the remainder of the proof we assume that $M$ is non-nilpotent.

Denote by $s_1(t),\dots,s_m(t)$ the invariant factors of $M$, so that $s_{i+1}(t)$ divides $s_i(t)$ for all $i \in \lcro 1,m-1\rcro$.
For $k \in \lcro 1,m\rcro$, split $s_k(t)=p_k(t) q_k(t)$ where $p_k(t)$ is a monic power of $t$, and $t$ does not divide $q_k(t)$.
Note that $p_{i+1}(t)$ divides $p_i(t)$ for all $i \in \lcro 1,m-1\rcro$.
Since $M$ has at least $n$ Jordan cells of size $1$ for the eigenvalue $0$,
there are at least $n$ integers $k$ such that $p_k(t)=t$, and as $M$ has $2n$ columns it follows that
$\underset{k=1}{\overset{m}{\sum}} \deg q_k(t) \leq 2n-n=n$; in particular:
\begin{enumerate}[(i)]
\item  There are  at most $n$ integers $k$ such that $q_k(t)$ is non-constant;
\item One has $m \geq n$;
\item One has $\deg q_m(t) \leq 1$.
\end{enumerate}

Now, we denote by $a$ the least integer $k$ such that $\deg p_k(t)=1$, and we set $r:=m-a$.
Hence, $p_a(t)=\cdots=p_m(t)=t$ and $\deg p_k(t)>1$ for all $k \in \lcro 1,a-1\rcro$.
Note that $r+1$ is the number of Jordan cells of size $1$ of $M$ for the eigenvalue $0$, whence
$r+1 \geq n$.

From there, we split the discussion into two subcases.

\noindent \textbf{Case 1:} One has $\deg q_k(t) \neq 1$ for all $k \in \lcro 1,m\rcro$.
Since $M$ is not nilpotent, this yields an integer $b \in \lcro 1,m\rcro$ such that
$\deg q_b(t) \geq 2$ and $q_k(t)=1$ for all $k \in \lcro b+1,m\rcro$.
Then,
$$M \simeq \underbrace{C\bigl(p_a(t)\bigr) \oplus C\bigl(p_{a-1}(t)\bigr) \oplus \cdots \oplus C\bigl(p_1(t)\bigr) \oplus 
C\bigl(q_1(t)\bigr) \oplus \cdots \oplus C\bigl(q_b(t)\bigr)}_N \oplus 0_{r},$$
It is clear that $N$ is well-partitioned with associated polynomials $p_a(t),\dots,p_1(t),q_1(t),\dots,q_b(t)$,
and the result follows by taking $q:=0$ and $\alpha:=1_\F$.

\noindent \textbf{Case 2:} Assume now that $\deg q_k(t)=1$ for some $k \in \lcro 1,m\rcro$.
We denote respectively by $c$ and $d$ the least and greatest integer $k$ such that $\deg q_k(t)=1$. Then, for some $\alpha \in \F \setminus \{0\}$, we have $q_c(t)=\cdots=q_d(t)=t-\alpha$,
whereas $q_k(t)=1$ for all $k \in \lcro d+1,m\rcro$, and $\deg q_k(t) \geq 2$ for all $k \in \lcro 1,c-1\rcro$.
Hence,
$$M \simeq \underbrace{C\bigl(p_a(t)\bigr) \oplus C\bigl(p_{a-1}(t)\bigr) \oplus \cdots \oplus C\bigl(p_1(t)\bigr) \oplus 
C\bigl(q_1(t)\bigr) \oplus \cdots \oplus C\bigl(q_c(t)\bigr)}_N \oplus (\alpha I_{d-c})\oplus 0_{r}.$$
Again, $N$ is well-partitioned. If $d-c$ is even, the claimed result is then obtained
by noting that
$$d-c \leq -1+\sum_{i=1}^m \deg q_i \leq 2n-r-2 \leq r.$$

Assume now that $d-c$ is odd.
\begin{itemize}
\item If $c>1$ we note that
$$N':=C\bigl(p_a(t)\bigr) \oplus C\bigl(p_{a-1}(t)\bigr) \oplus \cdots \oplus C\bigl(p_1(t)\bigr) \oplus C\bigl(q_1(t)\bigr) \oplus \cdots \oplus C\bigl(q_{c-1}(t)\bigr)$$
is well-partitioned and
$$M \simeq N' \oplus (\alpha\, I_{d-c+1})\oplus 0_{r,}$$
and we conclude as in the preceding situation because now $d-c+1 \leq n-\deg q_1(t) \leq n-2 \leq r$.

\item If $c=1$, then $N':=C\bigl(p_{a-1}(t)\bigr) \oplus \cdots \oplus C\bigl(p_1(t)\bigr)$ is nilpotent,
$M \simeq N' \oplus (\alpha\, I_{d-c+1}) \oplus 0_{r+1}$, $d-c+1 \leq 2n-(r+1) \leq r+1$ and $d-c+1$ is even.
\end{itemize}
\end{proof}

Hence, we are reduced to proving the following result:

\begin{prop}\label{normalized3squarezero}
Let $N \in \Mat_p(\F)$ be a well-partitioned matrix (possibly void) and $q$ be a non-negative integer,
and assume that $M:=N \oplus I_{2q} \oplus 0_{2q}$ has trace zero.
Then, $M$ is the sum of three square-zero matrices.
\end{prop}

Indeed, let $A \in \Mat_n(\F)$ be a trace-zero matrix and set $M:=A\oplus 0_n$.
Then, by Lemma \ref{WPlemma}, we have
$$M\simeq N \oplus \alpha\,I_{2q} \oplus 0_r$$
for some non-negative integers $p,q,r$ with $2q \leq r$ and $p+2q+r=2n$, some non-zero scalar $\alpha$ and some matrix $N \in \Mat_p(\F)$ that is either nilpotent or well-partitioned.
Hence, $\alpha^{-1} M \simeq  \alpha^{-1}\,N \oplus I_{2q} \oplus 0_{2q} \oplus 0_{r-2q}$, and if we prove that $\alpha^{-1}M$
is the sum of three square-zero matrices then so is $M$ because any scalar multiple of a square-zero matrix has square zero.
Moreover, $\alpha^{-1}\,N$ is either nilpotent or similar to a well-partitioned matrix.
Finally, $0_{r-2q}$ is the sum of three square-zero matrices, and hence if Proposition \ref{normalized3squarezero} holds
then we will deduce that $M$ is the sum of three square-zero matrices.

Proposition \ref{normalized3squarezero} will be established thanks to the following series of lemmas.

\begin{lemma}
There is a square-zero matrix $S \in \Mat_{4q}(\F)$ such that
$$I_{2q}\oplus 0_{2q}-S\simeq \underset{k=1}{\overset{q}{\bigoplus}} \,C\bigl((t-k)^2\bigr) \oplus C\bigl((t+k-1)^2\bigr).$$
\end{lemma}

\begin{proof}
Note that $I_{2q} \oplus 0_{2q}$ is the sum of $q$ copies of $I_2 \oplus 0_2$.
Let $k \in \lcro 1,q\rcro$.
We claim that there exists a square-zero matrix $S_k \in \Mat_4(\F)$
such that
$$(I_2 \oplus 0_2)-S_k \simeq C\bigl((t-k)^2\bigr)\oplus C\bigl((t+k-1)^2\bigr).$$
If $2k.1_\F \neq 1_\F$ then this directly follows from Lemma \ref{proj+squarezero}
since
$$C\bigl((t(t-1)-k(k-1))^2\bigr) \simeq C\bigl((t-k)^2\bigr) \oplus C\bigl((t+k-1)^2\bigr).$$
Assume now that $2k.1_\F=1_\F$. Then, we know from Lemma \ref{proj+squarezero} that there is a square-zero matrix $T_k$ such that
$$I_1\oplus 0_1 -T_k \simeq C\bigl((t(t-1)-k(k-1))\bigr) =C\bigl((t-k)^2\bigr).$$
Hence,
$$(I_1\oplus 0_1)\oplus (I_1\oplus 0_1) -T_k\oplus T_k \simeq C\bigl((t-k)^2\bigr)\oplus C\bigl((t+k-1)^2\bigr).$$
As $I_2 \oplus 0_2$ is similar to $(I_1\oplus 0_1)\oplus (I_1\oplus 0_1)$, we obtain the claimed result.

From there, we deduce that
$$\underset{k=1}{\overset{q}{\bigoplus}} (I_2 \oplus 0_2) -\underset{k=1}{\overset{q}{\bigoplus}}S_k \simeq
\underset{k=1}{\overset{q}{\bigoplus}}\,
C\bigl((t-k)^2\bigr) \oplus C\bigl((t+k-1)^2\bigr).$$
Since $I_{2q} \oplus 0_{2q}$ is similar to $\underset{k=1}{\overset{q}{\bigoplus}} (I_2 \oplus 0_2)$ whereas
$\underset{k=1}{\overset{q}{\bigoplus}}S_k$ has square zero, the existence of the claimed matrix $S$ follows.
\end{proof}

\begin{lemma}
If $q.1_\F=0_\F$ then $\underset{k=1}{\overset{q}{\bigoplus}} C\bigl((t-k)^2\bigr) \oplus C\bigl((t+k-1)^2\bigr)$
is the sum of two square-zero matrices.
\end{lemma}

\begin{proof}
Assume that $q.1_\F=0_\F$. Then,
$$\underset{k=1}{\overset{q}{\bigoplus}} C\bigl((t+k-1)^2\bigr)
=\underset{k=0}{\overset{q-1}{\bigoplus}} C\bigl((t+k)^2\bigr)
=\underset{k=1}{\overset{q}{\bigoplus}} C\bigl((t+k)^2\bigr)$$
since the polynomials $(t+q)^2$ and $t^2$ are equal over $\F$.
Hence,
$$\underset{k=1}{\overset{q}{\bigoplus}} C\bigl((t-k)^2\bigr) \oplus C\bigl((t+k-1)^2\bigr)
\simeq \underset{k=1}{\overset{q}{\bigoplus}} C\bigl((t-k)^2\bigr) \oplus C\bigl((t+k)^2\bigr),$$
 and the matrix on the right-hand side is obviously similar to its opposite. The conclusion then follows from Corollary \ref{similaropposite}.
\end{proof}

From there, we can finally prove Proposition \ref{normalized3squarezero}.

\begin{proof}[Proof of Proposition \ref{normalized3squarezero}]
The proof is split into two cases.

\noindent \textbf{Case 1: $\tr N=0$.}

Then, $0=\tr M=\tr N+2q.1_\F=2q.1_\F$. Since $\charac(\F) \neq 2$ we find $q.1_\F=0_\F$.
The above two lemmas then show that $I_{2q}\oplus 0_{2q}$ is the sum of three square-zero matrices.
\begin{itemize}
\item If $N$ is nilpotent then it is the sum of two square-zero matrices.
\item If $N$ is well-partitioned, then
Lemma \ref{wellpartitioned-squarezero} yields a square-zero matrix $S \in \Mat_p(\F)$ such that
$N-S \simeq C(t^p)$ (we convene that $C(t^p)=0$ if $p=0$), and hence
$N-S$ is the sum of two square-zero matrices.
\end{itemize}
In any case, $N$ is the sum of three square-zero matrices.
Hence, so is $M=N \oplus (I_{2q}\oplus 0_{2q})$.

\vskip 2mm
\noindent \textbf{Case 2: $\tr(N) \neq 0$.}

In particular $p \geq 2$, $\tr N=-2q.1_\F$ and $q.1_\F \neq 0$.
Hence, by Lemma \ref{wellpartitioned-squarezero} we can find a square-zero matrix $S \in \Mat_p(\F)$ such that
$N-S \simeq C(t^{p-2}(t+q)^2) \simeq C(t^{p-2}) \oplus C((t+q)^2)$.
On the other hand, we have a square-zero matrix $S'$ such that
$$(I_{2q}\oplus 0_{2q})-S' \simeq \underset{k=1}{\overset{q}{\bigoplus}} C\bigl((t-k)^2\bigr) \oplus C\bigl((t+k-1)^2\bigr).$$
Hence, $S_0:=S \oplus S'$ has square zero and
$$M-S_0 \simeq C\bigl(t^{p-2}\bigr) \oplus C\bigl((t+q)^2\bigr) \oplus \underset{k=1}{\overset{q}{\bigoplus}} C\bigl((t-k)^2\bigr) \oplus C\bigl((t+k-1)^2\bigr).$$
Yet,
$$C\bigl((t+q)^2\bigr) \oplus \underset{k=1}{\overset{q}{\bigoplus}} C\bigl((t-k)^2\bigr) \oplus C\bigl((t+k-1)^2\bigr)
\simeq C(t^2) \oplus \underset{k=1}{\overset{q}{\bigoplus}} C\bigl((t-k)^2\bigr) \oplus C\bigl((t+k)^2\bigr),$$
whence
$$M-S_0 \simeq C(t^{p-2}) \oplus C(t^2) \oplus \underset{k=1}{\overset{q}{\bigoplus}} C\bigl((t-k)^2\bigr) \oplus C\bigl((t+k)^2\bigr),$$
and once more Corollary \ref{similaropposite} yields that $M-S_0$ is the sum of two square-zero matrices.
Hence, $M$ is the sum of three square-zero matrices.
\end{proof}

Thus, the proof of Theorem \ref{allfieldssquarezero} is now complete.

\section{Sums of three square-zero matrices over a field of characteristic $2$}\label{SectionSquareZeroCar2}

This section consists of a proof of Theorem \ref{carac2squarezero}.

Throughout this section, we assume that the field $\F$ has characteristic $2$.
The basic strategy is similar to the one of the preceding section, with a few different details.
First of all, we shall need the following basic lemma.

\begin{lemma}\label{2by2car2}
Let $\alpha \in \F$. Then, $\alpha\,I_2$ is the sum of three square-zero matrices.
\end{lemma}

\begin{proof}
Indeed,
$$\alpha \,I_2-\begin{bmatrix}
0 & 1 \\
0 & 0
\end{bmatrix}=\begin{bmatrix}
\alpha & 1 \\
0 & \alpha
\end{bmatrix}\simeq C\bigl((t-\alpha)^2\bigr) = C(t^2-\alpha^2),$$
whereas by Theorem \ref{sumoftwotheo} the matrix $C(t^2-\alpha^2)$ is the sum of two square-zero ones.
\end{proof}

The next step consists of the following lemma, which actually holds over any field.

\begin{lemma}\label{wellpartlemma2}
Let $A \in \Mat_n(\F)$. Assume that $A$ has at most one Jordan cell of size $1$ for each one of its eigenvalues
in $\F$ and that its minimal polynomial is not a power of some irreducible polynomial. Then, $A$ is similar to a well-partitioned matrix.
\end{lemma}

\begin{proof}
Let us choose an irreducible monic factor $p(t)$ of the minimal polynomial of $A$.
Let us denote by $a_1(t),\dots,a_r(t)$ the invariant factors of $A$,
so that $a_{i+1}(t)$ divides $a_i(t)$ for all $i \in \lcro 1,r-1\rcro$.
For all $i \in \lcro 1,r\rcro$, split $a_i(t)=b_i(t)c_i(t)$ where $b_i(t)$ is a power of $p(t)$
and $c_i(t)$ is coprime with $p(t)$. Set finally $u:=\max \{i \in \lcro 1,r\rcro : b_i(t) \neq 1\}$ and
$v:=\max \{i \in \lcro 1,r\rcro : c_i(t) \neq 1\}$. If $p(t)$ has degree $1$ then we know that $A$ has at most one Jordan
cell of size $1$ for the root of $p(t)$. Hence, $\deg b_i(t) \geq 2$ for all $i \in \lcro 1,u-1\rcro$.
Likewise $\deg c_i(t) \geq 2$ for all $i \in \lcro 1,v-1\rcro$.
Hence,
$$A':=C\bigl(b_u(t)\bigr) \oplus C\bigl(b_{u-1}(t)\bigr) \oplus \cdots \oplus C\bigl(b_1(t)\bigr)\oplus C\bigl(c_1(t)\bigr)\oplus \cdots \oplus C\bigl(c_{v-1}(t)\bigr) \oplus C\bigl(c_v(t)\bigr)$$
is well-partitioned, and $A \simeq A'$.
\end{proof}

\begin{lemma}\label{carac2irr}
Let $M\in \Mat_n(\F)$ be a trace-zero matrix whose minimal polynomial is a power of some
irreducible monic polynomial. Then, $M$ is the sum of three square-zero matrices.
\end{lemma}

\begin{proof}
We write the minimal polynomial of $M$ as $p(t)^r$, where $p(t)$ is an irreducible monic polynomial, and
$r$ is a positive integer.
If $p(t)=t$, then $M$ is nilpotent and we already know that it is the sum of two square-zero matrices.
In the rest of the proof, we assume that $p(t) \neq t$.
We distinguish between two cases.

\noindent \textbf{Case 1: $\tr(p(t))=0$.} \\
Denote by $d$ the degree of $p(t)$.
Let $k$ be a positive integer. Then, $\tr(p(t)^k)=k\tr p(t)=0$.
The last column of $C(p(t)^k)$ then reads $\begin{bmatrix}
Y \\
0
\end{bmatrix}$ for some $Y \in \F^{kd-1}$. Setting $S:=
\begin{bmatrix}
0_{(kd-1) \times (kd-1)} & Y \\
0_{1 \times (kd-1)} & 0
\end{bmatrix}$, we see that $C(p(t)^k)-S$ is the transpose of a Jordan cell for the eigenvalue $0$.
Hence, $C(p(t)^k)-S$ is the sum of two square-zero matrices, and it follows that $C(p(t)^k)$ is the sum of three such matrices.
Using the rational canonical form, we deduce that $M$ is the sum of three square-zero matrices.

\vskip 3mm
\noindent \textbf{Case 2: $\tr(p(t))\neq 0$.} \\
As $\tr M=0$ and $\F$ has characteristic $2$, evenly many invariant factors of $M$ are odd powers of $p(t)$.
Hence, we have a splitting
$$M \simeq A_1 \oplus \cdots \oplus A_r \oplus B_1 \oplus \cdots \oplus B_s,$$
where each $A_i$ equals $C(p^k)$ for some even positive integer $k$, and each $B_j$ equals $C(p^k)\oplus C(p^l)$
for some pair $(k,l)$ of odd positive integers.
However, for every even integer $k>0$, we see that $p^k$ is an even polynomial (since $\F$ has characteristic $2$)
and hence $C(p^k)$ is the sum of two square-zero matrices.
To conclude, we take an arbitrary pair $(k,l)$ of odd positive integers and we prove that $C(p^k)\oplus C(p^l)$
is the sum of three square-zero matrices.
The matrix $S:=\begin{bmatrix}
0_{kd} & 0_{kd \times ld} \\
F_{ld,kd} & 0_{ld}
\end{bmatrix}$ of $\Mat_{(k+l)d}(\F)$ has square-zero, and one sees that $\bigl(C(p^k)\oplus C(p^l)\bigr)-S$
is a good cyclic matrix with characteristic polynomial $p^{k+l}$.
Since $k+l$ is even and $\F$ has characteristic $2$, the polynomial $p^{k+l}$ is an even one,
whence $\bigl(C(p^k)\oplus C(p^l)\bigr)-S$ is the sum of two square-zero matrices. This completes the proof.
\end{proof}

From there, the proof of Theorem \ref{carac2squarezero} can be completed swiftly.

\begin{proof}[Proof of Theorem \ref{carac2squarezero}]
Let $M$ be a trace-zero matrix of $\Mat_n(\F)$.
By pairing Jordan cells of size $1$ of $M$ associated to the same eigenvalue in $\F$,
we find a non-negative integer $p$, scalars $\alpha_1,\dots,\alpha_q$ (with $p+2q=n$) and a matrix $N \in \Mat_p(\F)$ such that
$$M \simeq N \oplus \alpha_1\,I_2\oplus \cdots \oplus \alpha_q\,I_2,$$
and $N$ has at most one Jordan cell of size $1$ for each one of its eigenvalues in $\F$.
Since $\F$ has characteristic $2$ we find $\tr N=\tr M=0$.
By Lemma \ref{2by2car2}, the conclusion will follow should we prove that $N$ is the sum of three square-zero matrices.
If the minimal polynomial of $N$ has a sole irreducible monic divisor, then this follows directly from Lemma \ref{carac2irr}.
Assume otherwise. We get from Lemma \ref{wellpartlemma2} that $N$ is similar to a trace-zero well-partitioned matrix $N'$.
By Lemma \ref{wellpartitioned-squarezero}, there is a square-zero matrix $S \in \Mat_p(\F)$ such that $N'-S \simeq C(t^p)$, so that
$N'-S$ is the sum of two square-zero matrices. Hence, $N'$ is the sum of three square-zero matrices,
and we conclude that so is $N$. Hence, $M$ is the sum of three square-zero matrices.
\end{proof}

\section{On sums of three idempotent matrices over a field of characteristic $2$}\label{SectionIdempotentCar2}

Throughout this section, we assume that $\F$ has characteristic $2$.
Our aim is to prove Theorem \ref{carac2idempotents}.
To start with, we recall the following result from \cite{dSPidem2} (Theorem 5 in that article):

\begin{theo}\label{sumof2idemcar2}
A matrix $A$ of $\Mat_n(\F)$ is the sum of two idempotent matrices if and only if,
for every eigenvalue $\lambda$ of $A$ in $\overline{\F} \setminus \{0,1\}$, all the Jordan
cells attached to $\lambda$ are even-sized.
\end{theo}

\begin{lemma}\label{3by33idemlemma}
Let $\alpha \in \F \setminus \{0_\F,1_\F\}$. Then, $\alpha\,I_2\oplus 0_1$ is the sum of three idempotents.
\end{lemma}

\begin{proof}
Note that $M:=\alpha\,I_2\oplus 0_1$ is similar to $N:=\begin{bmatrix}
\alpha & 0 & 0 \\
1 & 0 & 0 \\
0 & 0 & \alpha
\end{bmatrix}$.
Set $P:=\begin{bmatrix}
0 & 0 & 0 \\
0 & 1 & 0 \\
0 & 1 & 0
\end{bmatrix}$ and note that $P$ is idempotent and
$N-P=\begin{bmatrix}
\alpha & 0 & 0 \\
1 & 1 & 0 \\
0 & 1 & \alpha
\end{bmatrix} \simeq C\bigl((t-\alpha)^2(t-1)\bigr) \simeq C(t-1) \oplus C\bigl((t-\alpha\bigr)^2\bigr)$.
Then, by Theorem \ref{sumof2idemcar2}, $N-P$ is the sum of two idempotents, whence $N$ is the sum of three idempotents.
We conclude that $M$ is the sum of three idempotents.
\end{proof}

\begin{lemma}\label{3idemwellpart}
Let $A \in \Mat_n(\F)$ be a well-partitioned matrix with $\tr A \in \{0_\F,1_\F\}$.
Then, $A$ is the sum of three idempotent matrices.
\end{lemma}

\begin{proof}
Denote by $m$ the number of polynomials associated with $A$.
Set $\lambda:=\tr A-(m-1).1_\F \in \{0_\F,1_\F\}$.
By Lemma \ref{wellpartitioned-idempotent}, we can find an idempotent $P \in \Mat_n(\F)$
such that $A-P \simeq C\bigl(t^{n-1}(t-\lambda)\bigr)$. Since $\lambda \in \{0_\F,1_\F\}$, Theorem
\ref{sumof2idemcar2} yields that $A-P$ is the sum of two idempotent matrices.
\end{proof}

\begin{lemma}\label{3idempowerirr}
Let $A \in \Mat_n(\F)$, and assume that the minimal polynomial of $A$ is a power of some monic irreducible polynomial $p(t)$.
Assume also that $\tr(A) \in \{0_\F,1_\F\}$. Assume finally that if $p(t)=t-\lambda$, then $A$
has at most one Jordan cell of size $1$ for the eigenvalue $\lambda$.
Then, $A$ is the sum of three idempotents.
\end{lemma}

\begin{proof}
For every even integer $r$, we know from Theorem \ref{sumof2idemcar2} that $C(p(t)^r)$ is the sum of two idempotent matrices, and its trace is obviously zero.
Hence, we lose no generality in assuming that the invariant factors of $A$ are odd powers of $p(t)$.
We split the discussion into two cases.

\noindent \textbf{Case 1: $\tr p(t) \in \{0_\F,1_\F\}$.} \\
Let $k \geq 1$ be a positive integer. Denote by $d$ the degree of $p(t)$.
Then, $\tr p(t)^k=k \tr p(t) \in \{0_\F,1_\F\}$.
Let $Y \in \F^{kd-1}$, and set $P:=\begin{bmatrix}
0_{(kd-1) \times (kd-1)} & Y \\
0_{1 \times (kd-1)} & 1
\end{bmatrix}$, which is an idempotent matrix of $\Mat_{kd}(\F)$.
Then, $C(p^k)-P$ has trace $k \tr(p(t))-1 \in \{0_\F,1_\F\}$, and it is actually a companion matrix.
Obviously, for any monic polynomial $q$ with degree $kd$ and trace $k \tr\bigl(p(t)\bigr)-1$
the vector $Y$ can be chosen so as to have $C(p^k)-P=C(q)$.
Choosing $q=(t-1) t^{kd-1}$ if $k \tr(p(t))-1=1$, and $q=t^{kd}$ if $k \tr\bigl(p(t)\bigr)-1=0$,
we know from Theorem \ref{sumof2idemcar2} that $C(q)$ is the sum of two idempotents,
and it follows that $C(p^k)$ is the sum of three idempotents.
Using the rational canonical form, we conclude that $M$ is the sum of three idempotents.

\vskip 3mm
\noindent \textbf{Case 2: $\tr p(t)\not\in \{0_\F,1_\F\}$.} \\
As $\tr A \in \{0_\F,1_\F\}$, there are evenly many invariant factors of $A$ (remember that all of them are odd powers of $p(t)$).
To conclude, it suffices to take an arbitrary pair $(k,l)$ of odd positive integers, with $k \leq l$, and to prove that $B:=C\bigl(p(t)^k\bigr)\oplus C\bigl(p(t)^l\bigr)$ is the sum of three idempotent matrices,
unless $k=l=1$ and $p=t-\lambda$ for some $\lambda \in \F \setminus \{0_\F,1_\F\}$.
Assume indeed that we do not simultaneously have $k=l=1$ and  $p=t-\lambda$ for some $\lambda \in \F \setminus \{0_\F,1_\F\}$.

Denote by $d$ the degree of $p(t)$. For a positive integer $i$, denote by $D_i$ the diagonal matrix of $\Mat_i(\F)$
with all diagonal entries zero except the last one, which equals $1$. 
Assume first that $d \geq 2$.
Set $S_1:=D_d \oplus \cdots \oplus D_d$ (with $k+l$ copies of $D_d$) and
$$S_2:=\begin{bmatrix}
0_{kd} & 0_{kd \times ld} \\
F_{ld,kd} & 0_{ld}
\end{bmatrix}.$$
Since $d \geq 2$, one sees that $S_1+S_2$ is idempotent. On the other hand, one sees that
$B-S_1-S_2$ is a good cyclic matrix with characteristic polynomial $\bigl(p(t)+t^{d-1}\bigr)^{k+l}$.
Since $k+l$ is even, we deduce from Theorem \ref{sumof2idemcar2} that $B-(S_1+S_2)$ is the sum of two idempotent matrices.

Assume now that $d=1$, so that $p(t)=t-\lambda$ for some $\lambda \in \F$.
Then, $S:=\begin{bmatrix}
D_k & 0_{k \times l} \\
F_{l,k} & D_l
\end{bmatrix}$ is idempotent because $l>1$.
Moreover, one sees that $B-S$ is a good cyclic matrix with characteristic polynomial $(t-\lambda)^{k+l-2}(t-\lambda+1)^2$,
and again by Theorem \ref{sumof2idemcar2} it is the sum of two idempotent matrices.

In any case, we conclude that $C\bigl(p(t)^k\bigr)\oplus C\bigl(p(t)^l\bigr)$ is the sum of three idempotent matrices, which completes the proof.
\end{proof}

\begin{cor}\label{atmost1cell3idem}
Let $A \in \Mat_n(\F)$, and assume that $A$ has at most one Jordan cell of size $1$ for each one if its eigenvalues in $\F$.
Assume also that $\tr A \in \{0_\F,1_\F\}$.
Then, $A$ is the sum of three idempotent matrices.
\end{cor}

\begin{proof}
If the minimal polynomial of $A$ is not a power of some irreducible monic polynomial, then
Lemma \ref{wellpartlemma2} shows that $A$ is similar to a well-partitioned matrix, and by Lemma \ref{3idemwellpart}
$A$ is the sum of three idempotent matrices.

Otherwise, Lemma \ref{3idempowerirr} shows that $A$ is the sum of three idempotent matrices.
\end{proof}

Now, we are ready to complete the proof of Theorem \ref{carac2idempotents}.

Let $A \in \Mat_n(\F)$ be with $\tr A \in \{0_\F,1_\F\}$.
By pairing Jordan cells of size $1$ associated to the same eigenvalue in $\F$, we find a decomposition
$$A \simeq N \oplus \alpha_1\,I_2 \oplus \cdots \oplus \alpha_q\,I_2$$
in which $\alpha_1,\dots,\alpha_q$ are scalars, and $N \in \Mat_{n-2q}(\F)$ has at most one Jordan cell of size $1$
for each one of its eigenvalues in $\F$. Note that $\tr N=\tr A \in \{0_\F,1_\F\}$.
By Corollary \ref{atmost1cell3idem}, $N$ is the sum of three idempotent matrices.
On the other hand, for all $k \in \lcro 1,q\rcro$, we know from Lemma \ref{3by33idemlemma} that
$\alpha_k\,I_2 \oplus 0_1$ is the sum of three idempotent matrices. Hence,
$A \oplus 0_q$ is the sum of three idempotent matrices, QED.

\section{A note on linear combinations of idempotent matrices}\label{SectionIdempotentLC}

We recall the following terminology from \cite{dSPidempotentLC}:

\begin{Def}
Let $\alpha_1,\dots,\alpha_k$ be scalars. A matrix $M \in \Mat_n(\F)$ is called an
$(\alpha_1,\dots,\alpha_k)$-composite when there are idempotents $P_1,\dots,P_k$ of $\Mat_n(\F)$ such that
$M=\underset{i=1}{\overset{k}{\sum}} \alpha_i\,P_i$.
\end{Def}

Our aim here is to prove the following result:

\begin{theo}\label{idemLCtheo}
Let $\F$ be an arbitrary field and
let $\alpha \in \F \setminus \{0\}$. Then, for all $A \in \Mat_n(\F)$, there exist scalars $\beta,\gamma$
such that $A \oplus \alpha I_n$ is an $(\alpha,\beta,\gamma)$-composite.
\end{theo}

The motivation for proving Theorem \ref{idemLCtheo} is the following corollary,
which will be derived in the next section:

\begin{cor}\label{idemLCCor}
Let $V$ be an infinite-dimensional vector space over $\F$.
Let $u$ be a finite-rank endomorphism of $V$, and let $\alpha \in \F$. Then, $f:=\alpha\,\id_V+u$
is a linear combination of three idempotent endomorphisms of $V$.
\end{cor}

Now, we prove Theorem \ref{idemLCtheo}. As we will see, it is a mostly straightforward adaptation of the line of reasoning from
\cite{dSPidempotentLC}.

Our first lemma is obtained by following the proof of Lemma 15 of \cite{dSPidempotentLC}:

\begin{lemma}\label{compositelemma}
Let $(\alpha,\beta,\gamma)\in \F^3$ be such that $\alpha \neq \beta$ and $\alpha \neq 0$, and let
$(q,r,s)\in \N^3$ be with $s>0$.
Then, there exist a monic polynomial $p(t) \in \F[t]$ of degree $s$ such that $\tr p(t) \neq \gamma$, together with a scalar $\alpha'$
such that $C\bigl(p(t)\bigr)\oplus \alpha\,I_q\oplus \beta\,I_r$ is an $(\alpha,\alpha')$-composite.
\end{lemma}

Similarly, the next lemma is obtained by following the details of the proof\footnote{Note that the statement of Lemma 13 of \cite{dSPidempotentLC} is partly incorrect because it does not take into account the possibility
that the matrix $A$ be diagonalizable with exactly two eigenvalues.
Nevertheless, this has no impact on the validity of the rest of the results from \cite{dSPidempotentLC}.} of Lemma 13 of \cite{dSPidempotentLC}
(indeed, in the notation of that proof we have $\alpha=\alpha_1$ for $A:=M \oplus \alpha\,I_n$):

\begin{lemma}\label{idemLCwellpartlemma}
Let $M \in \Mat_n(\F)$ and $\alpha \in \F \setminus \{0\}$.
If $M$ is not triangularizable with sole eigenvalue $\alpha$,
then there exist a triple $(p,q,r)$ of non-negative integers such that $p+q+r=2n$,
a scalar $\beta \neq \alpha$ and
a matrix $N \in \Mat_p(\K)$ such that $M \oplus \alpha\,I_n \simeq N \oplus \alpha\,I_q \oplus \beta\,I_r$
and either $N=0$ or $N$ is very well-partitioned.
\end{lemma}

Now, we can finish the proof of Theorem \ref{idemLCtheo}:

\begin{proof}[Proof of Theorem \ref{idemLCtheo}]
Let $A \in \Mat_n(\F)$. Assume that $A$ is triangularizable with sole eigenvalue $\alpha$.
Hence, $M:=A \oplus \alpha\, I_n$ is also triangularizable with sole eigenvalue $\alpha$.
Then, $M-\alpha\,I_{2n}$ is nilpotent, and hence by Proposition 15 of \cite{dSPidem2} it is a $(1,-1)$-composite. 
Thus, $M$ is an $(\alpha,1,-1)$-composite. 

Assume now that $A$ is not triangularizable with sole eigenvalue $\alpha$.
Then, by Lemma \ref{idemLCwellpartlemma} there exist a scalar $\beta \neq \alpha$, a triple $(p,q,r)$ of non-negative integers, and
a matrix $N \in \Mat_p(\K)$ such that $A \oplus \alpha\,I_n \simeq A':=N \oplus \alpha\, I_q \oplus \beta \,I_r$
and either $N=0$ or $N$ is very well-partitioned.
If $N=0$ then $A \oplus \alpha\,I_n$ is an $(\alpha,\beta)$-composite, and hence an $(\alpha,\beta,1)$-composite.
Assume now that $N$ is very well-partitioned. By Lemma \ref{compositelemma}, there exist a monic
polynomial $u(t) \in \F[t]$ of degree $p$ and a scalar $\alpha'$
such that $\tr u(t) \neq \tr N$ and
$C\bigl(u(t)\bigr)\oplus \alpha\,I_q\oplus \beta\,I_r$ is an $(\alpha,\alpha')$-composite.
By Lemma \ref{verywellpartitioned-idempotent}, there exist a scalar $\delta$ and an idempotent $P \in \Mat_p(\F)$ such that
$N-\delta P \simeq C\bigl(u(t)\bigr)$.
Set $\widetilde{P}:=P \oplus 0_{q+r}$, which is idempotent. Then,
$$A'-\delta \widetilde{P} \simeq C\bigl(u(t)\bigr) \oplus \alpha\, I_q \oplus \beta\, I_r$$
is an $(\alpha,\beta)$-composite, and hence $A'$ is an $(\alpha,\beta,\delta)$-composite.
Therefore, $A \oplus \alpha\,I_n$ is an $(\alpha,\beta,\delta)$-composite. This completes the proof.
\end{proof}

\section{Application to the decomposition of endomorphisms of an infinite-dimensional space}\label{infinitedim}

In this section, we prove Corollaries \ref{3squarezeroCor}, \ref{3squarezeroCorcar2}, \ref{3idemCorcar2} and \ref{idemLCCor}.

\subsection{Proof of Corollary \ref{3squarezeroCor}}

We can choose a finite-dimensional linear subspace $W$ of $V$ such that $\im u \subset W$ and $\Ker u+W=V$.
Denote by $u'$ the endomorphism of $W$ induced by $u$, and set $n:=\dim W$.
Since $V$ is infinite-dimensional we can split $V=W \oplus W' \oplus W_0$,
where $\dim W'=n$ and $W' \oplus W_0 \subset \Ker u$.
Denote by $u''$ the endomorphism of $W\oplus W'$ induced by $u$.
Choose some square matrix $A$ that represents $u'$. Then, $\tr(A)=\tr(u)=0$ since $\im u \subset W$.
By Theorem \ref{allfieldssquarezero}, $A \oplus 0_n$ is the sum of three square-zero matrices, yielding
square-zero endomorphisms $a,b,c$ of $W \oplus W'$ such that $u''=a+b+c$.
Extending $a,b,c$ into endomorphisms $\widetilde{a},\widetilde{b},\widetilde{c}$ of $V$ which vanish everywhere on $W_0$,
we obtain that $\widetilde{a},\widetilde{b},\widetilde{c}$ have square zero and
$u=\widetilde{a}+\widetilde{b}+\widetilde{c}$.

\subsection{Proof of Corollary \ref{3squarezeroCorcar2}}

Set $f:=\alpha\,\id_V+u$.
We can choose a finite-dimensional linear subspace $W$ of $V$ such that $\im u \subset W$ and $\Ker u+W=V$.
Set $n:=\dim W$. Set $p:=0$ if $\tr u=n\alpha$, and $p:=1$ if $\tr u=(n+1)\alpha$.
Denote by $u'$ the endomorphism of $W$ induced by $u$.
Since $V$ is infinite-dimensional we can split $V=W \oplus W' \oplus W_0$,
where $\dim W'=n+p$ and $W' \oplus W_0 \subset \Ker u$.
Choose some square matrix $A$ that represents the endomorphism $u''$ of $W \oplus W'$ induced by $u$.
Then, $\tr A=\tr(u'')=\tr (u)$ since $\im u \subset W$.
Hence, $\tr(A+\alpha I_{n+p})=\tr(u)+(n+p)\alpha=0$, and $A+\alpha I_{n+p}$ represents $\alpha\,\id_{W \oplus W'}+u''$.
By Theorem \ref{carac2squarezero}, we obtain square-zero endomorphisms $a,b,c$ of $W \oplus W'$ such that
$f_{|W\oplus W'}=a+b+c$.
Next, as $W_0$ is infinite-dimensional, we can split it into $W_0=\underset{i \in I}{\bigoplus} P_i$
in which $P_i$ has dimension $2$ for all $i \in I$.
Let $i \in I$. Then, $f_{|P_i}=\alpha \,\id_{P_i}$ for all $i \in I$, which has trace $0$. By Theorem \ref{carac2squarezero},
there are square-zero endomorphisms $a_i,b_i,c_i$ of $P_i$ such that $f_{|P_i}=a_i+b_i+c_i$.
Define $\widetilde{a}$ as the endomorphism of $V$ that coincides with $a$ on $W\oplus W'$, and with $a_i$ on $P_i$ for all $i \in I$.
Likewise, define $\widetilde{b}$ and $\widetilde{c}$ from the data of $b,(b_i)_{i\in I}$ and $c,(c_i)_{i \in I}$, respectively.
Then, one checks that $\widetilde{a},\widetilde{b},\widetilde{c}$ have square zero and $f=\widetilde{a}+\widetilde{b}+\widetilde{c}$.

\subsection{Proof of Corollary \ref{3idemCorcar2}}

Set $f:=\alpha\,\id_V+u$.
We can choose a finite-dimensional linear subspace $W$ of $V$ such that $\im u \subset W$ and $\Ker u+W=V$.
Denote by $f'$ the endomorphism of $W$ induced by $f$, and set $n:=\dim W$.
Since $V$ is infinite-dimensional we can split $V=W \oplus W' \oplus W_0$,
where $\dim W'=n$ and $W' \oplus W_0 \subset \Ker u$.
Choose some square matrix $A$ that represents $f'$. Then, $\tr(A)=\tr(u)+n\alpha \in \{0_\F,1_\F\}$ since $\im u \subset W$.
The matrix $A \oplus \alpha\,I_{n}$ represents $f_{|W \oplus W'}$, and its trace belongs to $\{0_\F,1_\F\}$.
Hence, by Theorem \ref{allfieldssquarezero}, the endomorphism $f_{|W\oplus W'}$ is the sum of three idempotent
endomorphisms $p,q,r$ of $W \oplus W'$.

Let us extend $p,q$ into endomorphisms $\widetilde{p},\widetilde{q}$ of $V$ which vanish everywhere on $W_0$.
Let us extend $r$ into an endomorphism $\widetilde{r}$ of $V$ whose restriction to $W_0$ is $\alpha \id_{W_0}$.
Then, $\widetilde{p},\widetilde{q},\widetilde{r}$ are idempotent endomorphisms of $V$ and
$f=\widetilde{p}+\widetilde{q}+\widetilde{r}$.

\subsection{Proof of Corollary \ref{idemLCCor}}

If $\alpha=0$ then the result is a straightforward consequence of Theorem 1 of \cite{dSPidempotentLC}.
Assume now that $\alpha \neq 0$.

We can find a finite-dimensional linear subspace $W$ of $V$ such that $\im u \subset W$ and $W+\Ker u=V$.
Set $n:=\dim W$.
Hence, $W$ is stable under $f$ and we can denote by $f'$ the endomorphism of $W$ induced by $f$.
We split $V=W \oplus W' \oplus W_0$, where $W'\oplus W_0 \subset \Ker u$ and $\dim W=\dim W'$.
Denote respectively by $f_1$ and $f_0$ the endomorphisms of $W+W'$ and $W_0$ induced by $f$, and note that $f_0=\alpha \id_{W_0.}$
Choose a square matrix $A \in \Mat_n(\F)$ that represents $f'$. Then, $A \oplus \alpha\,I_n$ represents $f_1$.
By Theorem \ref{idemLCtheo}, there are scalars $\beta,\gamma$ and idempotent matrices $P,Q,R$ of $\Mat_{2n}(\F)$ such that $A\oplus \alpha I_n =\alpha P+\beta Q+\gamma R$.
This yields idempotent endomorphisms $p,q,r$ of $W+W'$ such that $f_1=\alpha\,p+\beta\,q+\gamma\,r$.
Hence, $f=\alpha\,(p \oplus \id_{W_0})+\beta\,(q\oplus 0_{W_0})+\gamma\,(r\oplus 0_{W_0})$.
Therefore, $f$ is a linear combination of three idempotent endomorphisms of $V$.

\appendix

\section{Appendix. On sums of two square-zero matrices}

This appendix consists in a short proof of Botha's theorem (that is, Theorem \ref{sumoftwotheo}).
First of all, the sufficient conditions:

\begin{lemma}\label{nilpotentlemma}
Let $n$ be a positive integer. Then, $C(t^n)$ is the sum of two square-zero matrices.
\end{lemma}

\begin{proof}
Define $A=(a_{i,j})$ and $B=(b_{i,j})$ in $\Mat_n(\F)$ by
$a_{i,j}=1$ if $i=j+1$ and $i$ is even, and $a_{i,j}=0$ otherwise,
and $b_{i,j}=1$ if $i=j+1$ and $i$ is odd, and $b_{i,j}=0$ otherwise.
One checks that $A^2=B^2=0$, while $A+B=C(t^n)$.
\end{proof}

\begin{lemma}\label{companionsquare}
Let $p$ be a non-constant monic polynomial.
Then,
$$C\bigl(p(t^2)\bigr) \simeq \begin{bmatrix}
0_n & C\bigl(p(t)\bigr) \\
I_n & 0_n
\end{bmatrix}.$$
\end{lemma}

\begin{proof}
Indeed, denote by $(e_1,\dots,e_{2n})$ the standard basis of $\F^{2n}$ and by
$u$ the endomorphism of $\F^{2n}$ associated with $\begin{bmatrix}
0_n & C\bigl(p(t)\bigr) \\
I_n & 0_n
\end{bmatrix}$ in it.
One checks that the matrix of $u$ in the basis $(e_1,e_{n+1},e_2,e_{n+2},\dots,e_n,e_{2n})$
is $C\bigl(p(t^2)\bigr)$.
\end{proof}

\begin{cor}\label{elemfactcor}
For all $M \in \Mat_n(\F)$, the invariant factors of
$$D(M):=\begin{bmatrix}
0_n & M \\
I_n & 0_n
\end{bmatrix}$$
are even polynomials.
\end{cor}

\begin{proof}
Note first that if $M=A_1 \oplus \dots \oplus A_p$ for some square matrices $A_1,\dots,A_p$,
then $D(M) \simeq D(A_1) \oplus \cdots \oplus D(A_p)$.
Moreover, for every invertible matrix $P \in \GL_n(\F)$,
we see that
$$D(PMP^{-1})=Q D(M) Q^{-1}$$
where $Q=P \oplus P \in \GL_{2n}(\F)$.
Using this last remark, we see that no generality is lost in assuming that
$M=C(p_1) \oplus \cdots \oplus C(p_r)$, where $p_1,\dots,p_r$ are monic polynomials,
and $p_{i+1}$ divides $p_i$ for all $i$ from $1$ to $r-1$.
Then,
$$D(M) \simeq D\bigl(C(p_1)\bigr) \oplus \cdots \oplus D\bigl(C(p_r)\bigr).$$
and by Lemma \ref{companionsquare}, this entails
$$D(M) \simeq D\bigl(C(p_1(t^2))\bigr)\oplus \cdots \oplus D\bigl(C(p_r(t^2))\bigr).$$
Obviously, $p_{i+1}(t^2)$ divides $p_i(t^2)$ for all $i$ from $1$ to $r-1$, whence the monic polynomials
$p_1(t^2),\dots,p_r(t^2)$ are the invariant factors of $D(M)$.
\end{proof}

\begin{cor}\label{evencor}
For every even monic polynomial $p(t)$, the matrix $C\bigl(p(t)\bigr)$ is the sum of two square-zero matrices.
\end{cor}

\begin{proof}
Indeed, $$\begin{bmatrix}
0_n & C\bigl(p(t)\bigr) \\
I_n & 0_n
\end{bmatrix}=\begin{bmatrix}
0_n & 0_n \\
I_n & 0_n
\end{bmatrix}+\begin{bmatrix}
0_n & C\bigl(p(t)\bigr) \\
0_n & 0_n
\end{bmatrix}$$
is the sum of two square-zero matrices, and hence the result follows directly from Lemma \ref{companionsquare}.
\end{proof}

\begin{cor}\label{sufficientcondcor}
For every even or odd monic polynomial $p(t)$, the matrix $C(p(t))$ is the sum of two square-zero matrices.
\end{cor}

\begin{proof}
Let $p(t)$ be an odd monic polynomial. Then, $p(t)=t^k q(t)$ for some positive integer $k$ and some even polynomial $q(t)$ such that $q(0) \neq 0$.
Then, $t^k$ and $q(t)$ are coprime, whence $C\bigl(p(t)\bigr) \simeq C(t^k)\oplus C\bigl(q(t)\bigr)$. By Lemma
\ref{nilpotentlemma} and Corollary \ref{evencor}, we conclude that $C\bigl(p(t)\bigr)$ is the sum of two square-zero matrices.
\end{proof}

\begin{Rem}
This last corollary can also be obtained with a similar proof as for Lemma \ref{nilpotentlemma}
(see the proof of Lemma 2 of \cite{Bothasquarezero}).
Yet, since we shall use Lemma \ref{companionsquare} once more later on, it was more efficient to prove Corollary \ref{sufficientcondcor} as we have done.
\end{Rem}

By Remarks \ref{remarksim} and \ref{remarkoplus}, we conclude that a matrix is the sum of two square-zero matrices if each one of its invariant factors is odd or even.
Now, we turn to the necessary conditions.

\begin{lemma}\label{commutelemma}
Let $a$ and $b$ be square-zero endomorphisms of a vector space over $\F$. Then, $a$ and $b$ commute with $(a+b)^2$.
\end{lemma}

\begin{proof}
Indeed, $(a+b)^2=ab+ba$, whence $a(a+b)^2=aba=(a+b)^2a$, and likewise $b$ commutes with $(a+b)^2$.
\end{proof}

\begin{cor}\label{invertiblepartcor}
Let $u$ be an endomorphism of a finite-dimensional vector space over $\F$.
Denote by $u_i$ the invertible part of $u$ in its Fitting decomposition.
If $u$ is the sum of two square-zero operators, then so is $u_i$.
\end{cor}

\begin{proof}
Denote by $n$ the dimension of the domain of $u$. Since $2n \geq n$, we know that $\im u^{2n}$
is the domain of $u_i$. Assume now that $u=a+b$ for some square-zero endomorphism $a$ and $b$ of the domain of $u$.
By Lemma \ref{commutelemma}, both $a$ and $b$ commute with $u^2$, whence both of them stabilize $\im (u^2)^n$.
Denoting by $a'$ and $b'$ the endomorphisms of $\im u^{2n}$ that are induced by $a$ and $b$, respectively,
we see that $(a')^2=(b')^2=0$ and $u_i=a'+b'$.
\end{proof}

\begin{lemma}\label{CNlemma}
Let $u$ be an automorphism of a finite-dimensional vector space $V$ over $\F$.
Assume that $u$ is the sum of two square-zero endomorphisms of $V$.
Then, the invariant factors of $u$ are even polynomials.
\end{lemma}

\begin{proof}
Let $a$ and $b$ be square-zero endomorphisms of $V$ such that $u=a+b$.
Since $\im a \subset \Ker a$ and $\im b \subset \Ker b$, we have
$\dim \Ker a \geq \frac{\dim V}{2}$ and $\dim \Ker b \geq \frac{\dim V}{2}$.
On the other hand, since $a+b$ is injective we have $\Ker a \cap \Ker b=\{0\}$.
It follows that $n \geq \dim \Ker a +\dim \Ker b$. We deduce that
$\dim \Ker a=\dim \Ker b=\frac{\dim V}{2}$, that $\Ker a=\im a$, $\Ker b=\im b$, and $V=\Ker a \oplus \Ker b$.

Choose a basis $(e_1,\dots,e_n)$ of $\Ker b$. Then, $(a(e_1),\dots,a(e_n))$ is a basis of $\im a=\Ker a$, whence
$(e_1,\dots,e_n,a(e_1),\dots,a(e_n))$ is a basis of $V$. In that basis, the matrices that represent $a$ and $b$ are, respectively,
$$A=\begin{bmatrix}
0_n & 0_n \\
I_n & 0_n
\end{bmatrix} \quad \text{and} \quad B=\begin{bmatrix}
0_n & M \\
0_n & 0_n
\end{bmatrix} \quad \text{for some $M \in \Mat_n(\F)$.}$$
Hence,
$$A+B=\begin{bmatrix}
0_n & M \\
I_n & 0_n
\end{bmatrix}.$$
By Corollary \ref{elemfactcor}, all the invariant factors of $A+B$ are even polynomials,
which proves our result.
\end{proof}

We can finish the proof of Theorem \ref{sumoftwotheo}.
Let $u$ be an endomorphism of a finite-dimensional vector space which splits into the sum of two square-zero endomorphisms. Then,
each invariant factor of $u$ is the product of a power of $t$ (possible $t^0$) with either $1$ or an invariant factor of the invertible part $u_i$ in the Fitting decomposition.
By Corollary \ref{invertiblepartcor} and Lemma \ref{CNlemma}, each invariant factor of $u_i$
is an even polynomial, whence every invariant factor of $u$ is an even or an odd polynomial.


\begin{thebibliography}{1}
\bibitem{Bothasquarezero}
J.D. Botha,
{Sums of two square-zero matrices over an arbitrary field,}
Linear Algebra Appl.
{\bf 436} (2012) 516--524.

\bibitem{Breaz}
S. Breaz, G. C\u{a}lug\u{a}reanu,
{Sums of nilpotent matrices,}
Linear Multilinear Algebra,
in press, doi: 10.1080/03081087.2016.1167816

\bibitem{Fillmore}
P.A. Fillmore,
{On similarity and the diagonal of a matrix,}
Amer. Math. Monthly
{\bf 76} (2) (1969) 167--169.

\bibitem{Roth}
W. Roth,
{The equations $AX-YB=C$ and $AX-XB=C$ in matrices,}
Proc. Amer. Math. Soc.
{\bf 3} (1952) 392--396.

\bibitem{dSPidempotentLC}
C. de Seguins Pazzis,
{On decomposing any matrix as a linear combination of three idempotents,}
Linear Algebra Appl.
{\bf 433} (2010) 843--855.

\bibitem{dSPidem2}
C. de Seguins Pazzis,
{On linear combinations of two idempotent matrices over an arbitrary field,}
Linear Algebra Appl.
{\bf 433} (2010) 625--636.

\bibitem{dSPidempotentsums}
C. de Seguins Pazzis,
{On sums of idempotent matrices over a field of positive characteristic,}
Linear Algebra Appl.
{\bf 433} (2010) 856--866.

\bibitem{dSPsumoftwotriang}
C. de Seguins Pazzis,
{Sums of two triangularizable quadratic matrices over an arbitrary field,}
Linear Algebra Appl.
{\bf 436} (2012) 3293--3302.

\bibitem{Takahashi}
K. Takahashi,
{On sums of three square-zero matrices,}
Linear Algebra Appl.
{\bf 306} (2000) 45--57.

\bibitem{WangWu}
J.-H. Wang, P.Y. Wu,
{Sums of square-zero operators,}
Studia Math.
{\bf 99} (1991) 115--127.
\end{thebibliography}
\end{document}